\documentclass{amsart}
\usepackage[utf8]{inputenc}
\usepackage[english]{babel}							
\usepackage[T1]{fontenc}\usepackage{amsthm}
\usepackage{amssymb}
\usepackage{amsmath}
\usepackage{graphicx} 
\usepackage{mathtools}
\usepackage{tikz}
\usepackage{ytableau,tikz,varwidth}
\usetikzlibrary{calc}
\usepackage{tikz-cd}
\usepackage{enumitem}

\tikzcdset{diagrams={arrows={shorten >=-0.5ex,shorten <=-0.5ex}}}

\newtheorem{theorem}{Theorem}[section]

\newtheorem{lemma}[theorem]{Lemma}
\theoremstyle{definition}
\newtheorem{definition}[theorem]{Definition}

\newtheorem{remark}[theorem]{Remark}

\newtheorem{prop}[theorem]{Proposition}

\newtheorem*{ack}{Acknowledgement}

\newcommand{\CommaPunct}{\mathpunct{\raisebox{0.0ex}{,}}}

\title{Non-normal edge rings satisfying $(S_{2})$-condition}
\author{Nayana Shibu Deepthi}
\address{Department of Pure and Applied Mathematics, Graduate School of Information Science and Technology, Osaka University, Suita, Osaka 565-0871, Japan}
\email{nayanasd@ist.osaka-u.ac.jp}

\begin{document}

\begin{abstract}
Let $G$ be a finite simple connected graph on the vertex set $V(G)=[d]=\{1,\dots ,d\}$, with edge set $E(G)=\{e_{1},\dots , e_{n}\}$. Let $K[\mathbf{t}]=K[t_{1},\dots , t_{d}]$ be the polynomial ring in $d$ variables over a field $K$. The edge ring of $G$ is the semigroup ring $K[G]$ generated by monomials $\mathbf{t}^{e}:=t_{i}t_{j}$, for $e=\{i,j\} \in E(G)$. In this paper, we will prove that, given integers $d$ and $n$, where $d\geq 7$ and $d+1\leq n\leq \frac{d^{2}-7d+24}{2}$, there exists a finite simple connected graph $G$ with $|V(G)|=d$ and $|E(G)|=n$, such that $K[G]$ is non-normal and satisfies $(S_{2})$-condition.
\end{abstract}

\maketitle

\section{Introduction}\label{sec:1}

Let $G$ be a finite simple connected graph on the vertex set $V(G)=[d]$ and let $E(G)=\{e_{1},\dots ,e_{n}\}$ be the edge set of $G$. Let us consider, $K[\mathbf{t}]=K[t_{1},\dots , t_{d}]$ to be the polynomial ring in $d$ variables over a field $K$. For an edge $e=\{i,j\}$ in $E(G)$, we define $\mathbf{t}^{e}:=t_{i}t_{j}$. The subring of $K[\mathbf{t}]$ generated by $\mathbf{t}^{e_{1}},\dots , \mathbf{t}^{e_{n}}$ is called the edge ring of $G$, denoted by $K[G]$. Let $\mathbf{e}_{1},\dots ,\mathbf{e}_{d}$ be the canonical unit coordinate vectors of $\mathbb{R}^{d}$ and for each $e=\{i,j\}\in E(G)$, we define $\rho(e):= \mathbf{e}_{i}+\mathbf{e}_{j}$. Let $S_{G}$ be the affine semigroup generated by $\rho(e_{1}),\dots , \rho(e_{n})$. Then, the edge ring $K[G]$ is the affine semigroup ring of $S_{G}$. 

For an affine semigroup $S\subset \mathbb{N}^{d}$, let $K[S]$ be the affine semigroup ring of $S$. While studying for a characterization of the Cohen-Macaulay affine semigroup ring, Goto and Watanabe \cite{GW} have defined an extension $S^{\prime}$ of $S$ and claimed that the condition, $S^{\prime}=S$ is the necessary and sufficient condition for $K[S]$ to be Cohen-Macaulay. Trung and Hoa \cite{TH} presented a counterexample and also demonstrated that $S^{\prime}=S$ is insufficient to establish the Cohen-Macaulayness of $K[S]$. They have also provided an additional topological condition on $C_{S}$, the convex rational polyhedral cone spanned by $S$ in $\mathbb{Q}^{d}$ and characterized the Cohen-Macaulayness of $K[S]$. Sch\"{a}fer and Schenzel \cite[ Theorem 6.3]{SS} claimed that the condition $S^{\prime}=S$ corresponds to the Serre's condition ($S_{2}$). For the introduction of Serre's condition ($S_{2}$), readers may refer to \cite[Section 2]{BH}.

The Cohen-Macaulayness of the edge ring $K[G]$ in terms of the corresponding graph $G$ has been a subject of extensive research. Given that the edge ring $K[G]$ is an affine semigroup ring, it is known from \cite[Theorem 1]{H} that, if $K[G]$ is normal then $K[G]$ is Cohen-Macaulay.
Ohsugi and Hibi \cite{OH} have characterized the normality of an edge ring in terms of its graph. At about the same time, Simis-Vasconcelos-Villarreal independently came to the same conclusion and reported it in \cite{SVV}. Recall from \cite[Theorem 2.2.22]{BH} that, the edge ring $K[G]$ is normal if and only if $K[G]$ satisfies Serre's conditions $(R_{1})$ and $(S_{2})$. In \cite[Theorem 2.1]{HKa}, Hibi and Katth\"{a}n have characterized the edge rings satisfying $(R_{1})$-condition. Note that Serre's condition $(S_{2})$ is a necessary condition for $K[G]$ to be Cohen-Macaulay. Based on these insights, Higashitani and Kimura \cite{HK} have provided the necessary condition for an edge ring to satisfy the $(S_{2})$-condition.

The main theorem that we will prove in this paper is as follows:

\begin{theorem}\label{thm:main}
Given integers $d$ and $n$ such that, $d\geq 7$ and $d+1\leq n\leq \frac{d^{2}-7d+24}{2}$, there exists a finite simple connected graph $G$ with $|V(G)|=d$ and $|E(G)|=n$ such that, the edge ring $K[G]$ is non-normal and satisfies $(S_{2})$-condition.
\end{theorem}

A detailed explanation of why the quadratic expression $\frac{d^2-7d+24}{2}$ appears in the main theorem is provided in Section~\ref{sec:6}.

This paper is organized as follows. In Section~\ref{sec:2}, we revisit some basic prerequisite definitions and results that will be encountered throughout this article. Section~\ref{sec:3} deals with introduction of the graph $G_{a,b}$, whose edge ring $K[G_{a,b}]$ is non-normal. In this section, we further prove that the edge ring $K[G_{a,b}]$ satisfies $(S_{2})$-condition. Section~\ref{sec:4} focuses on the step-wise removal of edges from $G_{a,b}$, such that each new graph obtained per step also satisfies both non-normality and $(S_{2})$-condition. In Section~\ref{sec:5}, we prove that any addition of new edges to the graph $G_{a,b}$, either affects the non-normality of the edge ring or leads to the violation of $(S_{2})$-condition. We conclude the article with Section~\ref{sec:6}, which deals with supporting evidence for Theorem~\ref{thm:main} and the conclusions.

\begin{ack}
I would like to thank Professor Akihiro Higashitani, my research supervisor, for introducing the problem and for his valuable comments that greatly improved the manuscript. I would also like to extend my gratitude to the anonymous reviewer for the attentive reading of the manuscript and for the insightful comments and suggestions. 
\end{ack}


\section{Preliminaries}\label{sec:2}

Let $G$ be a finite simple connected graph on the vertex set $V(G)=[d]$, with edge set $E(G)=\{e_{1},\dots ,e_{n}\}$. Let us consider $K[\mathbf{t}]=K[t_{1},\dots ,t_{d}]$ to be the polynomial ring in $d$ variables over a field $K$. For an edge $e=\{i,j\}\in E(G)$, we define $\mathbf{t}^{e}:=t_{i}t_{j}.$ The subring of $K[\mathbf{t}]$ generated by $\mathbf{t}^{e_{1}},\dots , \mathbf{t}^{e_{n}}$ is called the {\em edge ring} of $G$ and let it be denoted as $K[G]$. 

We consider $\mathbf{e}_{1},\dots ,\mathbf{e}_{d}$ to be the canonical unit coordinate vectors of $\mathbb{R}^{d}$. For some edge $e=\{i,j\}\in E(G)$,  define $\rho(e):= \mathbf{e}_{i}+\mathbf{e}_{j}$. Let $\mathcal{A}_{G}:=\{\rho(e)\colon e\in E(G)\}$ and let $S_{G}$ be the affine semigroup generated by $\rho(e_{1}),\dots , \rho(e_{n})$. We can express, $S_{G}:=\mathbb{Z}_{\geq 0}\mathcal{A}_{G}.$
Thus, the edge ring $K[G]$ is the affine semigroup ring of $S_{G}$. 

Consider $C_{G}$ to be the convex rational polyhedral cone spanned by $S_{G}$ in $\mathbb{Q}^{d}$. We may assume that $C_{G}$ is of dimension $d$ and let $\mathcal{F}(G)$ be the set of all facets of $C_{G}$. We define, $\overline{S}_{G}:=\mathbb{Q}_{\geq 0}\mathcal{A}_{G}\cap \mathbb{Z}\mathcal{A}_{G}.$ For any facet $F\in\mathcal{F}(G)$, we define
$$S_{F}:= S_{G}-S_{G}\cap F = \{\mathbf{x}\in \mathbb{Z}\mathcal{A}_{G}\colon \exists \ \mathbf{y}\in S_{G} \cap F \textrm{ such that } \mathbf{x}+\mathbf{y}\in S_{G}\},$$
and $S^{\prime}_{G}:=\bigcap\limits_{F\in \mathcal{F}(G)}S_{F}.$ By definition, $\mathcal{A}_{G}$ is said to be normal when we have $$\mathbb{Z}_{\geq 0}\mathcal{A}_{G}=\mathbb{Z}\mathcal{A}_{G}\cap\mathbb{Q}_{\geq 0}\mathcal{A}_{G},$$ that is, when $\overline{S}_{G}=S_{G}.$

A cycle is said to be {\em minimal} in $G$ if there exists no chord in it. Let $C$ and $C^{\prime}$ be two minimal cycles of $G$ with $V(C)\cap V(C^{\prime})=\emptyset$, if we have $i\in V(C)$ and $j\in V(C^{\prime})$, then $e=\{i,j\}\in E(G)$ is called a {\em bridge} between $C$ and $C^{\prime}$.

A pair of odd cycles $(C, C^{\prime})$ is called {\em exceptional} if $C$ and $C^{\prime}$ are minimal odd cycles in $G$ such that $V(C)\cap V(C^{\prime})=\emptyset $ and there exists no bridge connecting them. We say that a graph $G$ satisfies {\em odd cycle condition}, if for any two odd cycles $C$ and $C^{\prime}$ of $G$, either $V(C)\cap V(C^{\prime})\neq \emptyset $ or there exists a bridge between $C$ and $C^{\prime}$. In other words, the graph $G$ has no exceptional pairs.

\begin{theorem}[{From \cite[Theorem 2.2]{OH} and \cite[Theorem 1.1]{SVV}}]\label{thm:normal}
Let G be a finite simple graph. Then, the edge ring $K[G]$ is normal if and only if $G$ satisfies the odd cycle condition.
\end{theorem}

From all of the above observations, we have:
$$S_{G}=\overline{S}_{G}\iff K[G] \textrm{ is normal }\iff G \textrm{ satisfies odd cycle condition}.$$

Let us consider $U\subset V(G)$, and we define $G_{U}$ as the induced subgraph of $G$ with $V(G_{U})=U$ and $E(G_{U})=\{e\in E(G)\colon e\subset U\}.$ Let $i\in V(G)$ and we denote $G\backslash i$ as the induced subgraph of $G$ on the vertex set $V(G\backslash i)=V(G)\backslash\{i\}$. Consider a subset $T\subset V(G)$ and we define:
$$N(G;T):=\big\{v\in V(G)\colon \{v,w\}\in E(G)\textrm{ for some } w\in T\big\}.$$

A subset $T\subset V(G)$ is called {\em independent} if $\{t_{i},t_{j}\}\notin E(G)$ for any $t_{i},t_{j}\in T$. For an independent set $T\subset V(G)$, we define a {\em bipartite graph induced by $T$} as the graph on vertex set $T\cup N(G;T)$ with edge set $\big\{\{v,w\}\in E(G)\colon v\in T,\ w \in N(G;T)\big\}$.

Now, we will be looking at the facets of $C_{G}$. For that, let us look at some important definitions and theorems that have been discussed in \cite{OH}.

\begin{definition}
Let $G$ be a finite connected  simple graph with vertex set $V(G)$. A vertex $v\in V(G)$ is said to be {\em regular} in $G$ if every connected component of $G\backslash v$ contains at least one odd cycle. A non-empty set $T\subset V(G)$ is said to be {\em fundamental} in $G$ if all the conditions below are satisfied by $T$:
\begin{enumerate}
    \item $T$ is an independent set;
    \item the bipartite graph induced by $T$ is connected;
    \item either $T\cup N(G;T)=V(G)$ or every connected component of the graph $G_{V(G)\backslash T\cup N(G;T)}$ contains at least one odd cycle.
\end{enumerate}
\end{definition}

Facets of $C_{G}$ are given by the intersection of the half-spaces defined by the supporting hyperplanes of $C_{G}$, and was investigated by Ohsugi and Hibi \cite{OH}.

\begin{theorem}[{From \cite[Theorem 1.7]{OH}}]
Let $G$ be a finite connected simple graph on the vertex set $[d]$, containing at least one odd cycle. Then, all the supporting hyperplanes of $C_{G}$ are as follows:
\begin{enumerate}
    \item  $\mathcal{H}_{v} = \big\{(x_{1},\dots ,x_{d}) \in \mathbb{R}^{d} \colon x_{v} = 0\big\}$, where $v$ is a regular vertex in $G$.
    \item  $\mathcal{H}_{T} = \big\{(x_{1},\dots ,x_{d}) \in \mathbb{R}^{d} \colon \sum\limits_{ i\in T} x_{i} = \sum\limits_{j\in N(G;T)} x_{j}\big\}$, where $T$ is a fundamental set in $G$.
\end{enumerate}
\end{theorem}

In this paper, we denote $F_{v}$ and $F_{T}$ as the facets of $C_{G}$ corresponding to the hyperplanes $\mathcal{H}_{v}$ and $\mathcal{H}_{T}$ respectively.

With reference to the work of Ohsugi and Hibi \cite[Theorem 2.2]{OH}, normalization of the edge ring $K[G]$ can be expressed as $$\overline{S}_{G}=S_{G}\ + \ \mathbb{Z}_{\geq 0}\big\{\mathbb{E}_{C}+\mathbb{E}_{C^{\prime}}\colon (C,C^{\prime}) \textrm{ is exceptional in } G\big\},$$ where for any odd cycle $C$, we define $\mathbb{E}_{C}:= \sum\limits_{i\in V(C)}\mathbf{e}_{i} .$ We observe that, $$2\big(\mathbb{E}_{C}+\mathbb{E}_{C^{\prime}}\big)=\Bigg(\displaystyle{\sum\limits_{e\in E(C)}}\rho(e)+\displaystyle{\sum\limits_{e^{\prime}\in E(C^{\prime})}}\rho(e^{\prime})\Bigg)\in S_{G}.$$

Through the work of Goto and Watanabe \cite{GW} and the investigation by Trung and Hoa \cite{TH}, the necessary and sufficient conditions for $K[G]$ to be Cohen-Macaulay were established. Followed by this, Sch\"{a}fer and Schenzel \cite[ Theorem 6.3]{SS} stated that, the condition $S^{\prime}_{G}=S_{G}$ corresponds to the Serre's condition $(S_{2})$. 

In general,  $S_{G}\subset S^{\prime}_{G}\subset \overline{S}_{G}.$ Therefore, in order to prove that the edge ring $K[G]$ satisfies $(S_{2})$-condition, it is enough to show that for any $\alpha\in \overline{S}_{G}\backslash S_{G}$, the element $\alpha\notin S^{\prime}_{G}$. This implies that $S^{\prime}_{G}\subset S_{G}$ and therefore $S^{\prime}_{G}=S_{G}$.

 In \cite{HK}, Higashitani and Kimura have provided the necessary condition that a graph $G$ has to hold in order to satisfy $(S_{2})$-condition. 

\begin{theorem} [{From \cite[Theorem 4.1]{HK}}]\label{thm:S2}
Let $G$ be a finite simple connected graph. Suppose that, there exists an exceptional pair $(C,C^{\prime})$ satisfying the following two conditions:

\begin{enumerate}
    \item for each regular vertex $v\in V(G)\backslash [V(C)\cup V(C^{\prime})]$ in $G$, both $C$ and $C^{\prime}$ belong to the same connected component of the graph $G\backslash v ;$
    \item for each fundamental set $T \in G$ with $[V(C)\cup V(C^{\prime})]\cap [T\cup N(G;T)]=\emptyset$, both $C$ and $C^{\prime}$ belong to the same connected component of $G_{V(G)\backslash (T\cup N(G;T))}$.
\end{enumerate}

Then, $\mathbb{E}_{C}+\mathbb{E}_{C^{\prime}}\in S_{G}^{\prime}$. In particular, $S_{G}\neq S_{G}^{\prime}$.
\end{theorem}


\section{The Graph $G_{a,b}$}\label{sec:3}

This section explicitly deals with the construction and study of a special graph $G_{a,b}$, whose edge ring $K[G_{a,b}]$ is non-normal. We conclude the section by stating and proving a proposition that, the edge ring $K[G_{a,b}]$ satisfies $(S_{2})$-condition.

Let $G_{a,b}$ be a simple finite connected graph with $|V(G_{a,b})|=d=a+b+1$, where $3\leq a\leq b$. 
We construct the graph $G_{a,b}$ (Figure~\ref{fig:Gab}) such that, it is formed by the union of two complete graphs $K_{a+1}$ and $K_{b+1}$ with exactly one common vertex.
Let us consider the vertex set $V(G_{a,b})=V(K_{a+1})\cup V(K_{b+1})$, 
such that $V(K_{a+1})=\{u_{1},\dots , u_{a},w\}$ and $V(K_{b+1})=\{v_{1},\dots , v_{b},w\}$.

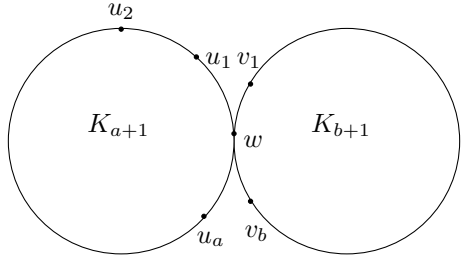
\begin{figure}[ht]
\centering
\begin{tikzpicture}
\filldraw[color=black!100, fill=white!5,  thin](6,3) circle (1.5);
\filldraw[color=black!100, fill=white!5, thin](3,3) circle (1.5);
\filldraw [black] (4,4.12) circle (0.8pt);
\filldraw [black] (3,4.49) circle (0.8pt);
\filldraw [black] (4.1,2) circle (0.8pt);
\filldraw [black] (4.72,2.2) circle (0.8pt);
\filldraw [black] (4.72,3.76) circle (0.8pt);
\filldraw [black] (4.5,3.1) circle (0.8pt);
\filldraw [black] (4.8,2) node[anchor=north] {$v_{b}$};
\filldraw [black] (4.7,4.3) node[anchor=north] {$v_{1}$};
\filldraw [black] (4.3,4.3) node[anchor=north] {$u_{1}$};
\filldraw [black] (4.5,3) node[anchor=west] {$w$};
\filldraw [black] (5.4,3.2) node[anchor=west] {$K_{b+1}$};
\filldraw [black] (3,4.5) node[anchor=south] {$u_{2}$};
\filldraw [black] (3.5,3.2) node[anchor=east] {$K_{a+1}$};
\filldraw [black] (4.5,1.7) node[anchor=east] {$u_{a}$};
\end{tikzpicture}
\caption{The graph $G_{a,b}$}
\label{fig:Gab}
\end{figure}

Any vertex $v\in V(G_{a,b})$ is regular in $G_{a,b}$. We observe that the fundamental sets in $G_{a,b}$ are : $$\{i\},\ \forall\ i\in V(G_{a,b}) \textrm{ and } \{i,j\},\ \forall\ \{i,j\}\notin E(G_{a,b}).$$

For any fundamental set $T\subset V(G_{a,b})$, let $H_{T}$ be the bipartite graph induced by $T$. We consider $K_{a}$ to be the induced subgraph of $G_{a,b}$ on vertex set $V(K_{a+1})\backslash \{w\}$, that is, the complete graph on $V(K_{a})=\{u_{1},\dots , u_{a}\}$. Similarly, $K_{b}$ is the induced subgraph of $G_{a,b}$ on the vertex set $V(K_{b+1})\backslash \{w\}$, which is the complete graph on $V(K_{b})=\{v_{1},\dots , v_{b}\}$. 

Let us denote $\mathcal{A}_{G_{a,b}}:=\{\rho(e)\colon e\in E(G_{a,b})\}$, $S :=S_{G_{a,b}}$ and $\overline{S}:=\overline{S_{G_{a,b}}}.$ We consider $C_{S}$ to be the convex rational polyhedral cone spanned by $\mathcal{A}_{G_{a,b}}$ in $\mathbb{Q}^{d}$, {\em i.e.}, $C_{S}:=C_{G_{a,b}}$. Let $\mathcal{F}(G_{a,b})$ be the set of all facets of $C_{S}$ and for any $F_{i}\in\mathcal{F}(G_{a,b}),$
$$S_{i}:= S-S\cap F_{i} = \{\mathbf{x}\in \mathbb{Z}\mathcal{A}_{G_{a,b}}\colon \exists\ \mathbf{y}\in S\cap F_{i} \textrm{ such that } \mathbf{x}+\mathbf{y}\in S\},$$
$$S^{\prime}:=\bigcap\limits_{F_{i}\in \mathcal{F}(G_{a,b})}S_{i}.$$

Now, let us look at the facets of $C_{S}$ in detail. We observe that, for any regular vertex $v\in V(G_{a,b})$,
$$S\cap F_{v}=\mathbb{Z}_{\geq 0}\mathcal{A}_{G_{a,b}\backslash v}.$$

Corresponding to each fundamental set in $G_{a,b},$ we have
$$S\cap F_{\{w\}}=\mathbb{Z}_{\geq 0}\mathcal{A}_{H_{\{w\}}},$$
$$S\cap F_{\{u_{i}\}}=\mathbb{Z}_{\geq 0}\mathcal{A}_{H_{\{u_{i}\}}\bigsqcup K_{b}},$$
$$S\cap F_{\{v_{j}\}}=\mathbb{Z}_{\geq 0}\mathcal{A}_{K_{a} \bigsqcup H_{\{v_{j}\}} },$$
$$S\cap F_{\{u_{i},v_{j}\}}=\mathbb{Z}_{\geq 0}\mathcal{A}_{H_{\{u_{i},v_{j}\}}},$$ 
where $1\leq i\leq a$ and $1\leq j\leq b.$ Note that, throughout this investigation, we are only concerned about the description of $S\cap F_{w}$, for a regular vertex $w$.

\begin{lemma}\label{lemma:1}
Let the pair of odd cycles $(C, C^{\prime})$ be exceptional in $G_{a,b}$. Consider vertices $v,w\in V(G_{a,b})$, where $w$ is the common vertex of $K_{a+1}$ and $K_{b+1}$. Let $\mathbf{e}_{v}$ and $\mathbf{e}_{w}$ be the canonical unit coordinate vectors of $\mathbb{R}^{d}$ corresponding to vertices $v$ and $w$ respectively. Then, 
$$\mathbb{E}_{C}\ +\ \mathbb{E}_{C^{\prime}}\ +\ \mathbf{e}_{v}\ + \ \mathbf{e}_{w} \in S.$$
\end{lemma}

\begin{proof}
We consider an exceptional pair $(C, C^{\prime})$ in $G_{a,b}$. Without loss of generality, let $C=\{u_{i_{1}},u_{i_{2}},u_{i_{3}}\}$ be a minimal odd cycle in $K_{a+1}$ and $C^{\prime}=\{v_{j_{1}},v_{j_{2}},v_{j_{3}}\}$ be a minimal odd cycle in $K_{b+1}$. 

Since $K_{a+1}$ and $K_{b+1}$ are complete graphs with common vertex $w$ and $(C, C^{\prime})$ is exceptional in $G_{a,b}$, we have $V(C)\cap V(C^{\prime})= \emptyset $ and $w\notin V(C)\cup V(C^{\prime})$.

Without loss of generality, we may assume $v \in V(K_{a+1})$. Given that $K_{a+1}$ and $K_{b+1}$ are complete graphs, for any $u_{i_{k}}\in V(C)$ that is distinct from $v$,  we have $\{v, u_{i_{k}}\}\in E(G_{a,b})$ and for any $v_{j_{k}}\in V(C^{\prime})$, we have $\{w, v_{j_{k}}\}\in E(G_{a,b})$. Suppose, we choose $u_{i_{1}}\neq v$. Then we can express

$\begin{array} {lcl}  \mathbb{E}_{C} + \mathbb{E}_{C^{\prime}} + \mathbf{e}_{v} +  \mathbf{e}_{w}& = & \sum\limits_{k=1}^{3}\mathbf{e}_{u_{i_{k}}}+\sum\limits_{k=1}^{3}\mathbf{e}_{v_{j_{k}}}+\mathbf{e}_{v} +  \mathbf{e}_{w}\\ \\
& = & (\mathbf{e}_{v}+ \mathbf{e}_{u_{i_{1}}})+ \sum\limits_{k=2}^{3}\mathbf{e}_{u_{i_{k}}}+\sum\limits_{k=1}^{2}\mathbf{e}_{v_{j_{k}}}+ (\mathbf{e}_{v_{j_{3}}}+  \mathbf{e}_{w}).
\end{array}$ 

This can be written as,
$$\mathbb{E}_{C} + \mathbb{E}_{C^{\prime}} + \mathbf{e}_{v} +  \mathbf{e}_{w}=\rho\big(\{v,u_{i_{1}}\}\big)+\rho\big(\{u_{i_{2}},u_{i_{3}}\}\big)+\rho\big(\{v_{j_{1}},v_{j_{2}}\}\big)+\rho\big(\{v_{j_{3}},w\}\big). $$

As we can see, the expression $\mathbb{E}_{C} + \mathbb{E}_{C^{\prime}} + \mathbf{e}_{v} +  \mathbf{e}_{w}$ can be written as a linear combination of some $\rho(e),$ where $e\in E(G_{a,b})$. Therefore, for any $v,w\in V(G_{a,b})$, we have $\big(\mathbb{E}_{C} +\mathbb{E}_{C^{\prime}} + \mathbf{e}_{v} +  \mathbf{e}_{w}\big) \in S.$ 
\end{proof}

Let us consider $\mathbf{x}=(x_{u_{1}},\dots , x_{u_{a}},x_{w},x_{v_{1}},\dots , x_{v_{b}})\in \mathbb{Z}^{d}_{\geq 0}$. We define a set,
$$A :=\bigg\{ \mathbf{x}\ \colon\ x_{w}=0,\ \displaystyle{\sum\limits_{u\in\{u_{1},\dots , u_{a}\}}}x_{u} \textrm{ is odd },\ \displaystyle{\sum\limits_{v\in\{v_{1},\dots , v_{b}\}}}x_{v} \textrm{ is odd }\bigg\}.$$

For the simple finite connected graph $G_{a,b}$ (Figure~\ref{fig:Gab}) and the set $A$ as defined above, we can state the following lemma.

\begin{lemma}\label{lemma:2}
$\overline{S}\subset S\cup A$.
\end{lemma}

\begin{proof}
Let $\alpha$ be an arbitrary element in $\overline{S}$. As we have seen in Section~\ref{sec:2}, the normalization of the edge ring $K[G_{a,b}]$ can be expressed as $$\overline{S}=S\ + \ \mathbb{Z}_{\geq 0}\big\{\mathbb{E}_{C}+\mathbb{E}_{C^{\prime}}\colon (C,C^{\prime}) \textrm{ is exceptional in } G_{a,b} \big\}.$$ Therefore, any $\alpha\in\overline{S}$ can be expressed as $\alpha =\beta + \gamma$, where we have $\beta\in S$ and $\gamma \in \mathbb{Z}_{\geq 0}\big\{\mathbb{E}_{C}+\mathbb{E}_{C^{\prime}}\colon (C,C^{\prime}) \textrm{ is exceptional in } G_{a,b}\big\}.$ If $\gamma = 0$, then $\alpha\in S$. So, let us consider the non-trivial case where $\gamma \neq 0$. Let $\alpha_{i}$, $\beta_{i}$ and $\gamma_{i}$ represent the $i^{\textrm{th}}$ coordinates of $\alpha$, $\beta$ and $\gamma$ respectively.

Since for any two (possibly identical) exceptional pairs $(C, C^{\prime}), (\overline{C}, \overline{C}^{\prime})$, it follows from the completeness of the graphs $K_{a+1}$ and $K_{b+1}$ that $\mathbb{E}_{C} + \mathbb{E}_{C^{\prime}}+ \mathbb{E}_{\overline{C}} + \mathbb{E}_{\overline{C}^{\prime}} \in S$. 
Therefore, without loss of generality, for an exceptional pair $(C, C^{\prime})$ in $G_{a,b}$, we may assume that $\gamma = \mathbb{E}_{C} + \mathbb{E}_{C^{\prime}}$. 

\noindent
\textbf{Case 1.}
Let $\alpha_{w}=0$. 

We have $\alpha_{w}=0$ and $\gamma \neq 0$ with $\gamma_{w}=0$. Therefore $\beta_{w}= 0$, that is, we are not considering any edge adjacent to the common vertex $w$. This assures that, both $\sum\limits_{u\in V(K_{a})}\beta_{u}$ and $\sum\limits_{v\in V(K_{b})}\beta_{v}$ have to be even. Hence, both $\sum\limits_{u\in V(K_{a})}\alpha_{u}$ and $\sum\limits_{v\in V(K_{b})}\alpha_{v}$ will be odd. Thus we have $\alpha\in A$ and therefore, $\alpha\in S\cup A$.

\noindent
\textbf{Case 2.}
Let $\alpha_{w}> 0$.

Consider an exceptional pair $(C, C^{\prime})$ in $G_{a,b}$. We have $\gamma =\mathbb{E}_{C} +\mathbb{E}_{C^{\prime}}$. The condition $\alpha_{w}> 0$ implies $\beta_{w}> 0$. This indicates that there must be at least one edge adjacent to the common vertex $w$, say $\{v,w\}$. For any exceptional pair $(C, C^{\prime})$ in $G_{a,b}$, by Lemma~\ref{lemma:1}, we have $\mathbb{E}_{C} +\mathbb{E}_{C^{\prime}} + \mathbf{e}_{v} +  \mathbf{e}_{w} \in S.$ Hence, $\alpha=\beta +\gamma\in S$. Thus, it proves that $\alpha\in S\cup A$. 
\end{proof}

All of the observations we have made so far lead us to the conclusion that, $$S\subset S^{\prime}\subset \overline{S}\subset S\cup A .$$

\begin{prop}\label{prop:1}
Let $G_{a,b}$ be a simple finite connected graph with $3\leq a\leq b$ and $|V(G_{a,b})|=a+b+1=d$, such that $G_{a,b}$ (Figure~\ref{fig:Gab}) consists of two complete graphs $K_{a+1}$ and $K_{b+1}$ joined at a common vertex $w$. Let $K[G_{a,b}]$ be the edge ring of the graph $G_{a,b}$. Then, $K[G_{a,b}]$ is non-normal and satisfies $(S_{2})$-condition.
\end{prop}

\begin{proof}
Since $G_{a,b}$ is the union of two complete graphs with a common vertex $w$, it is assured that all  the pairs of odd cycles of the form ($\{u_{i},u_{i+1},u_{i+2}\},\{v_{j},v_{j+1},v_{j+2}\}$) where $1\leq i\leq a-2$ and $1\leq j\leq b-2$ are exceptional. Therefore, the graph $G_{a,b}$ does not satisfy the odd cycle condition and by Theorem~\ref{thm:normal}, we conclude that the edge ring $K[G_{a,b}]$ is non-normal.

Let us consider an element from the set $A$, say $\alpha = (x_{u_{1}},\dots , x_{u_{a}},0,x_{v_{1}},\dots , x_{v_{b}})$ such that $\alpha\in \overline{S}\backslash S$. We have seen that the common vertex $w$ is regular in $G_{a,b}$ and corresponding to this regular vertex we have $S\cap F_{w}=\mathbb{Z}_{\geq 0 }\mathcal{A}_{G_{a,b}\backslash w}$. For any $\beta\in S\cap F_{w},$ let $\beta_{i}$ represent the $i^{\textrm{th}}$ coordinate of $\beta$. We observe that $\beta_{w}=0$ and both $\sum\limits_{i=1}^{a}\beta_{u_{i}}$, and $\sum\limits_{j=1}^{b}\beta_{v_{j}}$ are even. Therefore, for all $\beta \in S\cap F_{w}$, we have $\alpha +\beta \in A$ and not in $S$. Thus, there exists no $\beta \in S\cap F_{w}$, such that $\alpha +\beta \in S$. Therefore, $\alpha \notin S_{w}$. Hence,
$$\alpha\notin \bigcap\limits_{F_{i}\in\mathcal{F}(G_{a,b})}S_{i} =S^{\prime}.$$
This implies that, for any $\alpha\in\overline{S}\backslash S$, we have $\alpha\notin S^{\prime}$. Hence, $(\overline{S}\backslash S) \cap S^{\prime}= \emptyset$ and $S^{\prime}\subset S.$ We know that $S\subset S^{\prime}.$ Therefore, $S^{\prime}=S.$
\end{proof}

We have proved that $K[G_{a,b}]$ is non-normal and satisfies $(S_{2})$-condition. Now, we are interested in modifying the graph $G_{a,b}$, to see how the behaviour of the corresponding edge ring varies.


\section{Removing edges of $G_{a,b}$ and $(S_{2})$-condition}\label{sec:4}

In Section~\ref{sec:3}, we have studied the graph $G_{a,b}$ in detail. In this section, we will investigate whether we can remove edges of $G_{a,b}$ to obtain $\widetilde{G}$, a subgraph of $G_{a,b}$, with $V(\widetilde{G})=V(G_{a,b})$ and $|E(\widetilde{G})|=d+1$ such that the edge ring $K[\widetilde{G}]$ is non-normal and satisfies $(S_{2})$-condition. We will modify a method of eliminating edges of $G_{a,b}$ so that the common vertex $w$ remains regular in any graph created using this method. By the end of this section, we prove that $K[\widetilde{G}]$ is non-normal and satisfies $(S_{2})$-condition.

By eliminating one edge from the graph $G_{a,b}$ per step, we will gradually build up $\widetilde{G}$. First of all, we will be removing edges of the graph $K_{a+1}$ and will not alter the graph $K_{b+1}$. 
Let us denote $u_{0} \coloneqq w$ and $G^{u_{1}}_{0} \coloneqq G_{a,b}$. We construct a new subgraph of $G_{a,b}$
through an edge removal process such that for each $1\leq i\leq a-3$, 
\begin{itemize}
   \item we remove the edge $\{u_{0},u_{i}\}$ from $G^{u_{i}}_{0}$ to obtain $G^{u_{i}}_{i}$, and
    \item remove the edge $\{u_{i},u_{j}\}$ from $G^{u_{i}}_{j-1}$ to obtain $G^{u_{i}}_{j}$,  $\forall \ i+1\leq j\leq a-1$.
\end{itemize}
Let us denote $G^{u_{i}}_{0}:=G^{u_{i-1}}_{a-1}$, $\forall\ 2\leq i\leq a-2$. By construction, we observe that $$E(G^{u_{i}}_{0}) = E(G_{a,b})\backslash \bigcup\limits_{p=1}^{i-1}\bigg\{\{u_{p},u_{q}\}\colon q\neq p, 0\leq q \leq a-1\bigg\}, \hspace{0.2cm} \forall \ 2\leq i\leq a-2.$$
For $i=a-2$, the construction of the subgraph $G^{u_{i}}_{j}$, $a-2\leq j \leq a-1$ is as follows:
\begin{itemize}
    \item the subgraph $G^{u_{a-2}}_{a-2}$ is constructed by removing the edge $\{u_{0},u_{a-2}\}$ from $G^{u_{a-2}}_{0}$;
    \item the edge $\{u_{0},u_{a-1}\}$ is removed from $G^{u_{a-2}}_{a-2}$, to obtain the subgraph $G^{u_{a-2}}_{a-1}$.
\end{itemize}

\begin{remark}\label{rem:0.1}
For $1\leq i\leq a-2$, $i\leq j\leq a-1$ and $( i, j)\neq (a-2,a-1)$, we have
\vfill
$\begin{array} {lcl} E(G^{u_{i}}_{j}) & = & E(G_{a,b})\backslash \big\{\{u_{p},u_{q}\}\colon q\neq p, (p,q)\in [a-1] \times [i-1] \\ 
&  &  \hspace{1.5cm}\cup \{(0,1),(0,2),\dots ,(0,i),(i,i+1),\dots,(i,j)\}\big\}.
\end{array}$
\vfill
\noindent
In particular, if $1\leq p< q \leq a$, then $\{u_{p},u_{q}\}\in E(G^{u_{i}}_{j})$ if and only if one of the following cases happens:

\begin{enumerate}[label=(\roman*)]
    \item $p\geq i+1;$
    \item $p=i$, $j+1\leq q\leq a;$
    \item $p\leq i-1$, $q=a.$
\end{enumerate}
If $1\leq p< q < r \leq a$, then $u_{p}, u_{q}, u_{r}$ form a triangle in $G^{u_{i}}_{j}$ if and only if either of the following cases happens:

\begin{enumerate}[label=(\alph*)]
    \item $p=i$, $j+1\leq q;$
    \item $p\geq i+1$.
\end{enumerate}
Moreover, for $1 \leq t \leq a$, we have $\{u_{t},w\}\in E(G^{u_{i}}_{j})$ if and only if $t \geq i+1$.
\end{remark}

\begin{remark}\label{rem:0.2}
For $( i, j) = (a-2,a-1)$, we have 
\vfill
$\begin{array} {lcl} E(G^{u_{a-2}}_{a-1}) & = & E(G_{a,b})\backslash \big\{\{u_{p},u_{q}\}\colon q\neq p, (p,q)\in [a-1] \times [a-3] \\ 
&  &  \hspace{1.5cm}\cup \{(0,1),(0,2),\dots ,(0,a-1)\}\big\}.
\end{array}$
\vfill
\noindent
In particular, if $1\leq p< q $, then $\{u_{p},u_{q}\}\in E(G^{u_{a-2}}_{a-1})$ if and only if one of the following cases happens: 

\begin{enumerate}[label=(\roman*)]
    \item $p\leq a-1$, $q= a;$
    \item $( p, q) = (a-2,a-1)$.
\end{enumerate}
If $1\leq p< q < r \leq a$, then $u_{p}, u_{q}, u_{r}$ form a triangle in $G^{u_{a-2}}_{a-1}$ if and only if $( p,q,r) = (a-2,a-1,a)$.
Moreover, $\{u_{t},w\}\in E(G^{u_{a-2}}_{a-1})$ if and only if $t = a$.
\end{remark}

An example of the sequence of subgraphs $G^{u_{i}}_{j}$; $1\leq i\leq 2$ and $i\leq j\leq 3$ that is constructed from the graph $G_{4,3}$ using the edge removal process defined above is illustrated in Figure \ref{fig:rmveg1}. 

%
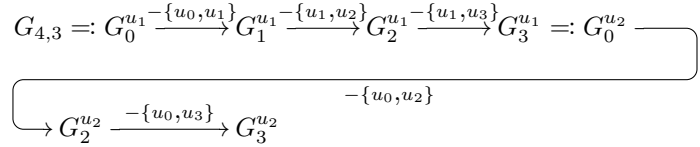
\begin{figure}[ht]
\begin{tikzcd}[arrow style=math font,cells={nodes={text height=2ex,text depth=1ex}}]
  {G_{4,3} \eqqcolon G^{u_{1}}_{0}} \arrow[r , "-\{u_{0} \CommaPunct u_{1}\}" ]
  \arrow[d, phantom, ""{coordinate, name=Z}]
  & G^{u_{1}}_{1} \arrow[r, "-\{u_{1} \CommaPunct u_{2}\}"]
    & G^{u_{1}}_{2} \arrow[r, "-\{u_{1} \CommaPunct u_{3}\}"]
      & G^{u_{1}}_{3}\eqqcolon G^{u_{2}}_{0} \arrow[dlll,
                 "-\{u_{0} \CommaPunct u_{2}\}" ,
                   rounded corners,
                    to path={ -- ([xshift=5ex]\tikztostart.east)
                    |- (Z) [near end]\tikztonodes
                    -| ([xshift=-3ex]\tikztotarget.west) -- (\tikztotarget)}] \\
    G^{u_{2}}_{2} \arrow[r, "-\{u_{0} \CommaPunct u_{3}\}"]
    & G^{u_{2}}_{3} 
\end{tikzcd}
\caption{A sequence of subgraphs constructed from $G_{4,3}$}
\label{fig:rmveg1}
\end{figure}

According to our construction, $V(G^{u_{i}}_{j})=V(G_{a,b})$, $\forall \ 1\leq i\leq a-2,\ i\leq j\leq a-1$. By the end of this entire process, we construct the subgraph of $G_{a,b}$, as shown in Figure~\ref{fig:remov}.

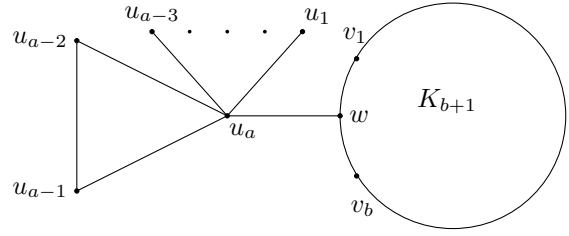
\begin{figure}[ht]
\centering
\begin{align*}
\begin{tikzpicture}
\filldraw[color=black!100, fill=white!5, thin](6,3) circle (1.5);
\draw[black, thin] (3,3) -- (1,2) -- (1,4)-- cycle;
\filldraw [black] (3,3) circle (0.8pt);
\filldraw [black] (1,2) circle (0.8pt);
\filldraw [black] (1,4) circle (0.8pt);
\filldraw [black] (4,4.12) circle (0.8pt);
\filldraw [black] (2,4.12) circle (0.8pt);
\filldraw [black] (3,4.12) circle (0.5pt);
\filldraw [black] (2.5,4.12) circle (0.5pt);
\filldraw [black] (3.5,4.12) circle (0.5pt);
\filldraw [black] (4.72,2.2) circle (0.8pt);
\filldraw [black] (4.72,3.76) circle (0.8pt);
\filldraw [black] (4.5,3) circle (0.8pt);
\draw[black, thin] (4.5,3) -- (3,3);
\draw[black, thin] (4,4.12) -- (3,3);
\draw[black, thin] (2,4.12) -- (3,3);
\filldraw [black] (4.8,2) node[anchor=north] {$v_{b}$};
\filldraw [black] (4.7,4.3) node[anchor=north] {$v_{1}$};
\filldraw [black] (4.2,4.1) node[anchor=south] {$u_{1}$};
\filldraw [black] (4.5,3) node[anchor=west] {$w$};
\filldraw [black] (5.4,3.2) node[anchor=west] {$K_{b+1}$};
\filldraw [black] (1,2) node[anchor=east] {$u_{a-1}$};
\filldraw [black] (1,4) node[anchor=east] {$u_{a-2}$};
\filldraw [black] (2,4.1) node[anchor=south] {$u_{a-3}$};
\filldraw [black] (2.9,2.8) node[anchor=west] {$u_{a}$};
\end{tikzpicture}
\end{align*}
\caption{The graph $G_{a-1}^{u_{a-2}}$}
\label{fig:remov}
\end{figure}

Let $C= \{u_{i_{1}}, u_{i_{2}},\dots , u_{i_{2l+1}}\}$ be an odd cycle in $G^{u_{i}}_{j}$, such that $2l+1 \geq 5$ and $ 1 \leq i_1 <i_{2} < \cdots < i_{2l+1}$. We have $\{u_{i_{1}}, u_{i_{2}}\} \in E(G^{u_{i}}_{j})$ and $ \{u_{i_{1}}, u_{i_{2l+1}}\}\in E(G^{u_{i}}_{j})$. Thus according to our construction, $\{ u_{i_{1}},u_{k}\}\in E(G^{u_{i}}_{j}),$ for all $i_{1}< k\leq i_{2l+1}$. Hence, for all $1\leq i\leq a-2$; $i\leq j\leq a-1$, the minimal odd cycles of $G^{u_{i}}_{j}$ are cycles of length three.

\begin{remark}\label{rem:rmv2}
Let $1\leq i\leq a-2$, $i\leq j\leq a-1$ be integers. Let $C= \{u_{i_{1}}, u_{i_{2}}, u_{i_{3}}\}$ be a triangle in $G^{u_{i}}_{j}$, where $1\leq i_{1}< i_{2} < i_{3}$. Then from Remarks \ref{rem:0.1} and \ref{rem:0.2}, $w$ is always adjacent to $ u_{i_{3}}$. Indeed, if $(i,j)\neq (a-2, a-1)$ then $w$ is even adjacent to both $u_{i_{2}}$ and $ u_{i_{3}}$. If $(i,j) = (a-2, a-1)$ then $(i_{1}, i_{2} , i_{3}) =(a-2, a-1, a)$ and $w$ is adjacent to $ u_{i_{3}} = u_{a}$.
\end{remark}


\begin{remark}\label{rem:rmv}
Let $1\leq i\leq a-2$, $i\leq j\leq a-1$ be integers.
Let $C$ and $\overline{C}$ be two (possibly identical) triangles in $G^{u_{i}}_{j}$.
Then by Remarks \ref{rem:0.1} and \ref{rem:0.2}, there is an edge connecting two vertices of $C$ and $\overline{C}$. 
Indeed, let $C= \{u_{i_{1}}, u_{i_{2}}, u_{i_{3}}\}$ and $\overline{C}= \{u_{j_{1}}, u_{j_{2}}, u_{j_{3}}\}$ where $1\leq i_{1}< i_{2} < i_{3}$, $1\leq j_{1}< j_{2} < j_{3}$. 
Then $i_{2}, j_{2} \geq i+1 $, so $u_{i_{2}}$ is adjacent to either $u_{j_{2}}$ or $u_{j_{3}}$.
\end{remark}

\begin{lemma}\label{lemma:1.1}
Let $1\leq i\leq a-2$, $i\leq j\leq a-1$ be integers.
Let $(C, C^{\prime})$ be an exceptional pair in $G^{u_{i}}_{j}$. If $\{w,v\}\in E(G^{u_{i}}_{j})$, then 
$$ \mathbb{E}_{C}\ +\ \mathbb{E}_{C^{\prime}}\ +\ \mathbf{e}_{w}\ + \ \mathbf{e}_{v} \in S_{G^{u_{i}}_{j}}.$$
\end{lemma}

\begin{proof}
Since $(C, C^{\prime})$ is an exceptional pair, we have $V(C)\cap V(C^{\prime})= \emptyset $ and $w\notin V(C)\cup V(C^{\prime})$.
By Remark \ref{rem:rmv}, we may assume that $V(C) = \{u_{i_{1}},u_{i_{2}},u_{i_{3}}\} \subset V(K_{a})$ and $V(C^{\prime}) = \{v_{j_{1}},v_{j_{2}},v_{j_{3}}\} \subset V(K_{b})$, where $i_{1}< i_{2} < i_{3}$, $ j_{1}< j_{2} < j_{3}$. 

\noindent
\textbf{Case 1.}
Let $v= u_{k}\in V(K_{a})$. 
We claim that $ \mathbb{E}_{C} + \mathbf{e}_{v} , \  \mathbb{E}_{C^{\prime}} + \mathbf{e}_{w}  \in S_{G^{u_{i}}_{j}}.$
By Remark \ref{rem:rmv2}, $w$ is adjacent to $v_{j_{3}}$, so 
$$\mathbb{E}_{C^{\prime}} + \mathbf{e}_{w} = \rho\big(\{w, v_{j_{3}}\}\big) + \rho\big(\{v_{j_{1}},v_{j_{2}}\}\big) \in S_{G^{u_{i}}_{j}}.$$
Since $\{u_{i_{1}},u_{i_{2}}\}$, $\{u_{i_{1}},u_{i_{3}}\} $ are edges and $i_{1}< i_{2} < i_{3}$, by Remarks \ref{rem:0.1} and \ref{rem:0.2}, $i_{1}\geq i$ and hence $i_{2}\geq i+1$. 
Since $\{w,u_{k}\}$ is an edge, by the same results, $k \geq i + 1$. 
Hence Remarks \ref{rem:0.1} and \ref{rem:0.2} imply that $u_{k}$ is adjacent to either $u_{i_{2}}$ or $u_{i_{3}}$. 
This implies $\mathbb{E}_{C} + \mathbf{e}_{v} \in S_{G^{u_{i}}_{j}}.$

\noindent
\textbf{Case 2.}
Let $v\in V(K_{b})$. We claim that $ \mathbb{E}_{C} + \mathbf{e}_{w}  , \  \mathbb{E}_{C^{\prime}} + \mathbf{e}_{v}  \in S_{G^{u_{i}}_{j}}.$
Since $K_{b}$ is complete, $\mathbb{E}_{C^{\prime}}  + \mathbf{e}_{v}  \in S_{G^{u_{i}}_{j}}.$ 
By Remark \ref{rem:rmv2}, $w$ is adjacent to $u_{i_{3}}$. 
Hence $ \mathbb{E}_{C} + \mathbf{e}_{w} \in S_{G^{u_{i}}_{j}}.$

\noindent
In both cases, we get the desired containment.
\end{proof}

For the graph $G^{u_{i}}_{j}$, where $1\leq i\leq a-2$, $i\leq j\leq a-1$ and the set $A$ as defined in Section \ref{sec:3}, we have the following lemma.

\begin{lemma}\label{lemma:2.1}
$\overline{S_{G^{u_{i}}_{j}}}\subset  S_{G^{u_{i}}_{j}}\cup A $, for all $1\leq i\leq a-2$ and $i\leq j\leq a-1$.
\end{lemma}

\begin{proof}
Let $\alpha$ be an arbitrary element in $\overline{S_{G^{u_{i}}_{j}}}$, where $1\leq i\leq a-2$, $i\leq j\leq a-1$. The normalization of the semigroup $S_{G^{u_{i}}_{j}}$ can be expressed as $$\overline{S_{G^{u_{i}}_{j}}}=S_{G^{u_{i}}_{j}}\ + \ \mathbb{Z}_{\geq 0}\big\{\mathbb{E}_{C}+\mathbb{E}_{C^{\prime}}\colon (C,C^{\prime}) \textrm{ is exceptional in } G^{u_{i}}_{j} \big\}.$$ 
Therefore, any $\alpha \in \overline{S_{G^{u_{i}}_{j}}}$ can be expressed as $\alpha =\beta + \gamma$, where $\beta\in S_{G^{u_{i}}_{j}}$ and $\gamma \in \mathbb{Z}_{\geq 0}\big\{\mathbb{E}_{C}+\mathbb{E}_{C^{\prime}}\colon (C,C^{\prime}) \textrm{ is exceptional in } G^{u_{i}}_{j}\big\}$.
If $\gamma = 0$, then $\alpha\in S_{G^{u_{i}}_{j}}$. 
So, let us consider the non-trivial case where $\gamma \neq 0$. Let $\alpha_{k}$, $\beta_{k}$ and $\gamma_{k}$ represent the $k^{\textrm{th}}$ coordinates of $\alpha$, $\beta$ and $\gamma$ respectively.

For any two (possibly identical) exceptional pairs $(C, C^{\prime}), (\overline{C}, \overline{C}^{\prime})$, using Remark \ref{rem:rmv} and from the completeness of $K_{b+1}$, we have $\mathbb{E}_{C} + \mathbb{E}_{C^{\prime}}+ \mathbb{E}_{\overline{C}} + \mathbb{E}_{\overline{C}^{\prime}} \in S_{G^{u_{i}}_{j}}$ for all $1\leq i\leq a-2$ and $i \leq j\leq a-1$.
Therefore, without loss of generality, for an exceptional pair $(C, C^{\prime})$ in $G^{u_{i}}_{j}$, we may assume that $\gamma = \mathbb{E}_{C} + \mathbb{E}_{C^{\prime}}$. 

\noindent
\textbf{Case 1.}
Let $\alpha_{w}=0$. 

We have $\gamma_{w}=0$ and $\alpha_{w}=0$. Therefore $\beta_{w}= 0$, that is, we are not considering any edge adjacent to the common vertex $w$. This assures that, both $\sum\limits_{u\in V(K_{a})}\beta_{u}$ and $\sum\limits_{v\in V(K_{b})}\beta_{v}$ have to be even. Hence, both $\sum\limits_{u\in V(K_{a})}\alpha_{u}$ and $\sum\limits_{v\in V(K_{b})}\alpha_{v}$ will be odd. Thus we have $\alpha\in A$.

\noindent
\textbf{Case 2.}
Let $\alpha_{w}> 0$.

The condition $\alpha_{w}> 0$ implies $\beta_{w}> 0$. 
This indicates that among the edges defining the vector $\beta$, there must be at least one edge adjacent to $w$, say $\{w,v\}$. 
For any exceptional pair $(C, C^{\prime})$ in $G^{u_{i}}_{j}$, by Lemma~\ref{lemma:1.1}, $ \mathbb{E}_{C} +\mathbb{E}_{C^{\prime}} + \mathbf{e}_{w} +  \mathbf{e}_{v} \in S_{G^{u_{i}}_{j}},$ and thus $\alpha=\beta +\gamma\in S_{G^{u_{i}}_{j}}$. 
\end{proof}

Therefore, for all $ 1\leq i\leq a-2$, and $i\leq j\leq a-1$, we can observe that, $$S_{G^{u_{i}}_{j}}\subset S_{G^{u_{i}}_{j}}^{\prime}\subset \overline{S_{G^{u_{i}}_{j}}}\subset  S_{G^{u_{i}}_{j}}\cup A .$$

\begin{prop}\label{prop:2}
The edge ring $K[G^{u_{i}}_{j}]$ of the graph $G^{u_{i}}_{j}$ is non-normal and satisfies $(S_{2})$-condition, for all $1\leq i\leq a-2$ and $i\leq j\leq a-1$.
\end{prop}

\begin{proof}
For any $1\leq k\leq b-2$, the pair $(\{u_{a-2},u_{a-1},u_{a}\},\{v_{k},v_{k+1},v_{k+2}\})$ is always exceptional in $G^{u_{i}}_{j}$. Hence, $K[G^{u_{i}}_{j}]$ is always non-normal.

Let us consider an element $\alpha\in \overline{S_{G_{j}^{u_{i}}}} \backslash S_{G_{j}^{u_{i}}}$. By Lemma \ref{lemma:2.1}, we have $\alpha \in A $ and $\alpha_{w}=0$. We observe that, the common vertex $w$ is regular in $G^{u_{i}}_{j}$. 
Hence, corresponding to $w$, we have $S_{G^{u_{i}}_{j}}\cap F_{w}:=\mathbb{Z}_{\geq 0 }\mathcal{A}_{G^{u_{i}}_{j}\backslash w}$.
For any $\beta\in S_{G_{j}^{u_{i}}}\cap F_{w},$ let $\beta_{k}$ be the $k^{\textrm{th}}$ coordinate of $\beta$. We observe that $\beta_{w}=0$ and both $\sum\limits_{i=1}^{a}\beta_{u_{i}}$, and $\sum\limits_{j=1}^{b}\beta_{v_{j}}$ are even. 
Therefore, for all $\beta \in S_{G_{j}^{u_{i}}}\cap F_{w}$, we have $\alpha +\beta \in A$ and not in $S_{G_{j}^{u_{i}}}$.
Thus, there exists no $\beta \in S_{G_{j}^{u_{i}}}\cap F_{w}$, such that $\alpha +\beta \in S_{G_{j}^{u_{i}}}$, and this implies that, $\alpha\notin S_{G_{j}^{u_{i}}}^{\prime}$.
As a result, we have $(\overline{S_{G_{j}^{u_{i}}}}\backslash S_{G_{j}^{u_{i}}}) \cap S^{\prime}_{G_{j}^{u_{i}}}= \emptyset$ and $S_{G_{j}^{u_{i}}}^{\prime}\subset S_{G_{j}^{u_{i}}}$. Therefore, $S^{\prime}_{G^{u_{i}}_{j}}=S_{G^{u_{i}}_{j}}.$
\end{proof}

Now, let us continue a similar edge removal process on $K_{b+1}$ and remove the maximum number of edges from $K_{b+1}$ resulting in the formation of the graph $\widetilde{G}$, as per our requirement.
Let $v_{0} \coloneqq w$ and $\widetilde{G}^{v_{1}}_{0}:=G^{u_{a-2}}_{a-1}$. We construct a new subgraph of $G_{a,b}$ through an edge removal process such that for each $1\leq i\leq b-3$, 
\begin{itemize}
   \item we remove the edge $\{v_{0},v_{i}\}$ from $\widetilde{G}^{v_{i}}_{0}$ to obtain $\widetilde{G}^{v_{i}}_{i}$, and
    \item remove the edge $\{v_{i},v_{j}\}$ from $\widetilde{G}^{v_{i}}_{j-1}$ to obtain $\widetilde{G}^{v_{i}}_{j}$,  $\forall \ i+1\leq j\leq b-1$.
\end{itemize}
Let us denote $\widetilde{G}^{v_{i}}_{0}:=\widetilde{G}^{v_{i-1}}_{b-1}$, $\forall\ 2\leq i\leq b-2$. By construction, we observe that $$E(\widetilde{G}^{v_{i}}_{0}) = E(\widetilde{G}^{v_{1}}_{0})\backslash \bigcup\limits_{p=1}^{i-1}\bigg\{\{v_{p},v_{q}\}\colon q\neq p, 0\leq q \leq b-1\bigg\}, \hspace{0.2cm} \forall \ 2\leq i\leq b-2.$$
For $i=b-2$, the construction of the subgraph $\widetilde{G}^{v_{i}}_{j}$, $b-2\leq j \leq b-1$ is as follows:
\begin{itemize}
    \item the subgraph $\widetilde{G}^{v_{b-2}}_{b-2}$ is constructed by removing the edge $\{v_{0},v_{b-2}\}$ from $\widetilde{G}^{v_{b-2}}_{0}$, and 
    \item the edge $\{v_{0},v_{b-1}\}$ is removed from $\widetilde{G}^{v_{b-2}}_{b-2}$, to obtain the subgraph $\widetilde{G}^{v_{b-2}}_{b-1}$.
\end{itemize}

As per construction, $V(\widetilde{G}^{v_{i}}_{j})=V(G_{a,b})$, $\forall \ 1\leq i\leq b-2,\ i\leq j\leq b-1$ and by the end of this removal procedure, we construct the graph depicted in Figure~\ref{fig:rmvd}.

\begin{figure}[ht]
\centering
\begin{align*}
\begin{tikzpicture}
\draw[black, thin] (3,3) -- (1,2) -- (1,4)-- cycle;
\draw[black, thin] (7,3) -- (9,2) -- (9,4)-- cycle;
\filldraw [black] (3,3) circle (0.8pt);
\filldraw [black] (1,2) circle (0.8pt);
\filldraw [black] (1,4) circle (0.8pt);
\filldraw [black] (7,3) circle (0.8pt);
\filldraw [black] (8,4.12) circle (0.8pt);
\draw[black, thin] (8,4.12) -- (7,3);
\filldraw [black] (6,4.12) circle (0.8pt);
\draw[black, thin] (6,4.12) -- (7,3);
\filldraw [black] (9,2) circle (0.8pt);
\filldraw [black] (9,4) circle (0.8pt);
\filldraw [black] (4,4.12) circle (0.8pt);
\filldraw [black] (2,4.12) circle (0.8pt);
\filldraw [black] (3,4.12) circle (0.5pt);
\filldraw [black] (2.5,4.12) circle (0.5pt);
\filldraw [black] (3.5,4.12) circle (0.5pt);
\filldraw [black] (6.5,4.12) circle (0.5pt);
\filldraw [black] (7.5,4.12) circle (0.5pt);
\filldraw [black] (7,4.12) circle (0.5pt);
\filldraw [black] (5,3) circle (0.8pt);
\draw[black, thin] (7,3) -- (3,3);
\draw[black, thin] (4.5,3) -- (3,3);
\draw[black, thin] (4,4.12) -- (3,3);
\draw[black, thin] (2,4.12) -- (3,3);
\filldraw [black] (7,3) node[anchor=north] {$v_{b}$};
\filldraw [black] (6,4.6) node[anchor=north] {$v_{1}$};
\filldraw [black] (9,2) node[anchor=west] {$v_{b-1}$};
\filldraw [black] (8,4.12) node[anchor=south] {$v_{b-3}$};
\filldraw [black] (9,4) node[anchor=west] {$v_{b-2}$};
\filldraw [black] (4.2,4.1) node[anchor=south] {$u_{1}$};
\filldraw [black] (5,3) node[anchor=south] {$w$};
\filldraw [black] (1,2) node[anchor=east] {$u_{a-1}$};
\filldraw [black] (1,4) node[anchor=east] {$u_{a-2}$};
\filldraw [black] (2,4.1) node[anchor=south] {$u_{a-3}$};
\filldraw [black] (2.9,2.8) node[anchor=west] {$u_{a}$};
\end{tikzpicture}
\end{align*}
\caption{The graph $\widetilde{G}$}
\label{fig:rmvd}
\end{figure}
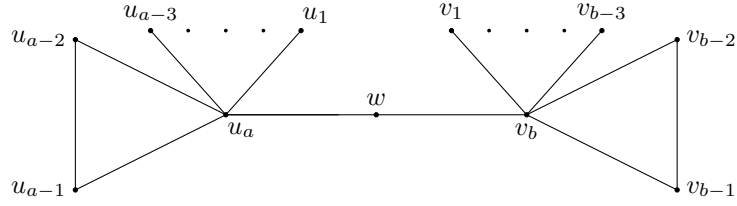

\begin{remark}\label{rem:0.3}
For $1\leq i\leq b-2$, $i\leq j\leq b-1$ and $( i, j)\neq (b-2,b-1)$, we have

\smallskip

$\begin{array} {lcl} E(\widetilde{G}^{v_{i}}_{j}) & = & E(\widetilde{G}^{v_{1}}_{0})\backslash \big\{\{v_{p},v_{q}\}\colon q\neq p, (p,q)\in [b-1] \times [i-1] \\ 
&  &  \hspace{1.5cm}\cup \{(0,1),(0,2),\dots ,(0,i),(i,i+1),\dots,(i,j)\}\big\}.
\end{array}$
\smallskip

\noindent
In particular, if $1\leq p< q \leq b$, then $\{v_{p},v_{q}\}\in E(\widetilde{G}^{v_{i}}_{j})$ if and only if one of the following cases happens:
\begin{enumerate}[label=(\roman*)]
    \item $p\geq i+1;$
    \item $p=i$, $j+1\leq q\leq b;$
    \item $p\leq i-1$, $q=b$.
\end{enumerate}
If $1\leq p< q < r \leq b$, then $v_{p}, v_{q}, v_{r}$ form a triangle in $\widetilde{G}^{v_{i}}_{j}$ if and only if either of the following cases happens:

\begin{enumerate}[label=(\alph*)]
    \item $p=i$, $j+1\leq q;$
    \item $p\geq i+1$.
\end{enumerate}
Moreover, for $1 \leq t \leq b$, we have $\{v_{t},w\}\in E(\widetilde{G}^{v_{i}}_{j})$ if and only if $t \geq i+1$.
\end{remark}

\begin{remark}\label{rem:0.4}
For $( i, j) = (b-2,b-1)$, we have 

\smallskip

$\begin{array} {lcl} E(\widetilde{G}^{v_{b-2}}_{b-1}) & = & E(\widetilde{G}^{v_{1}}_{0})\backslash \big\{\{v_{p},v_{q}\}\colon q\neq p, (p,q)\in [b-1] \times [b-3] \\ 
&  &  \hspace{1.5cm}\cup \{(0,1),(0,2),\dots ,(0,b-1)\}\big\}.
\end{array}$
\smallskip

\noindent
In particular, if $1\leq p< q $, then $\{v_{p},v_{q}\}\in E(\widetilde{G}^{v_{b-2}}_{b-1})$ if and only if one of the following cases happens: 

\begin{enumerate}[label=(\roman*)]
    \item $p\leq b-1$, $q= b;$
    \item $( p, q) = (b-2,b-1)$.
\end{enumerate}
If $1\leq p< q < r \leq b$, then $v_{p}, v_{q}, v_{r}$ form a triangle in $\widetilde{G}^{v_{b-2}}_{b-1}$ if and only if $( p,q,r) = (b-2,b-1,b)$.
Moreover, $\{v_{t},w\}\in E(\widetilde{G}^{v_{b-2}}_{b-1})$ if and only if $t = b$.
\end{remark}


\begin{remark}\label{rem:0.5}
From Remarks \ref{rem:0.3} and \ref{rem:0.4}, we see that the minimal odd cycles of $\widetilde{G}^{v_{i}}_{j}$ are triangles. Let $C$ be a triangle of $\widetilde{G}^{v_{i}}_{j}$, we claim that a vertex of $C$ is adjacent to $w$. 
If $V (C) \subseteq V (K_{a})$, as $\widetilde{G}^{v_{i}}_{j}$ is a subgraph of $\widetilde{G}^{v_{1}}_{0} = G^{u_{a-2}}_{a-1}$, we must have $C = \{u_{a-2}, u_{a-1}, u_{a}\}$.
In this case, $w$ adjacent to $u_{a}$. 

\noindent
If $C$ is a subgraph of $K_{b}$, let its vertices be $v_{j_{1}}, v_{j_{2}}, v_{j_{3}}$ where $1 \leq j_{1}< j_{2} < j_{3}\leq b$. 
Then $w$ is adjacent to $v_{j_{3}}$, as Remarks \ref{rem:0.3} and \ref{rem:0.4} implies that $j_{1} \geq i$ and $i+1 \leq j_{2} < j_{3}$. 
In both cases, a vertex of $C$ is adjacent to $w$.

\noindent
Moreover, for any two (possibly identical) triangles $C$ and $\overline{C}$ of $\widetilde{G}^{v_{i}}_{j}$, whose vertices are inside $K_{b}$, there is an edge of $\widetilde{G}^{v_{i}}_{j}$ connecting a vertex of $C$ to a vertex of $\overline{C}$.
\end{remark}

\begin{lemma}\label{lemma:1.2}
Let us consider an exceptional pair $(C, C^{\prime})$ in $\widetilde{G}^{v_{i}}_{j}$, where $1\leq i\leq b-2$ and $i\leq j\leq b-1$. 
If $\{w,v\}\in E(\widetilde{G}^{v_{i}}_{j})$, then we have
$$ \mathbb{E}_{C}\ +\ \mathbb{E}_{C^{\prime}}\ +\ \mathbf{e}_{w}\ + \ \mathbf{e}_{v} \in S_{\widetilde{G}^{v_{i}}_{j}}.$$
\end{lemma}

\begin{proof}
By Remark \ref{rem:0.5}, we may assume that $V (C) \subseteq V (K_{a})$ and $V (C^{\prime}) \subseteq V (K_{b})$. The same remark implies that $C = \{u_{a-2}, u_{a-1}, u_{a}\}$. Let the vertices of $C^{\prime}$ be $v_{j_{1}}, v_{j_{2}}, v_{j_{3}}$ where $1 \leq j_{1}< j_{2} < j_{3}\leq b$.

\noindent
\textbf{Case 1.}
$v\in V(K_{a})$. 
We claim that $ \mathbb{E}_{C} + \mathbf{e}_{v} , \  \mathbb{E}_{C^{\prime}} + \mathbf{e}_{w}  \in S_{\widetilde{G}^{v_{i}}_{j}}.$
Given $\{w,v\}\in E(\widetilde{G}^{v_{i}}_{j})$, as per the construction of $\widetilde{G}^{v_{i}}_{j}$, $v=u_{a}$. Since $C = \{u_{a-2}, u_{a-1}, u_{a}\}$, we see that $v$ is adjacent to both $u_{a-2}$ and $u_{a-1}$, so $\mathbb{E}_{C} + \mathbf{e}_{v} \in S_{\widetilde{G}^{v_{i}}_{j}}.$
By Remark \ref{rem:0.5}, $w$ is adjacent to a vertex of $C^{\prime}$, hence $\mathbb{E}_{C^{\prime}} + \mathbf{e}_{w} \in  S_{\widetilde{G}^{v_{i}}_{j}}.$

\noindent
\textbf{Case 2.}
$v = v_{k}\in V(K_{b})$. We claim that $ \mathbb{E}_{C} + \mathbf{e}_{w}  , \  \mathbb{E}_{C^{\prime}} + \mathbf{e}_{v}  \in S_{\widetilde{G}^{v_{i}}_{j}}.$
Since $w$ is adjacent to $u_{a}$, $\mathbb{E}_{C} + \mathbf{e}_{w} \in  S_{\widetilde{G}^{v_{i}}_{j}}$. Since $w$ is adjacent to $v_{k}$, by Remarks \ref{rem:0.3} and \ref{rem:0.4}, $k \geq i+1$. The same remarks imply that $j_{1} \geq i,\ j_{2} \geq i+1$. Hence $v_{k}$ is adjacent to either $v_{j_{2}}$ or $v_{j_{3}}$. This yields $\mathbb{E}_{C^{\prime}} + \mathbf{e}_{v}  \in S_{\widetilde{G}^{v_{i}}_{j}}$.

\noindent
In both cases, we get the desired containment.
\end{proof}

For the set $A$ as defined in Section \ref{sec:3} and the graph $\widetilde{G}^{v_{i}}_{j}$, where $1\leq i\leq b-2$, and $i\leq j\leq b-1$, we have the following lemma.

\begin{lemma}\label{lemma:2.2}
$\overline{S_{\widetilde{G}^{v_{i}}_{j}}}\subset  S_{\widetilde{G}^{v_{i}}_{j}}\cup A $, for all $1\leq i\leq b-2$ and $i\leq j\leq b-1$.
\end{lemma}


\begin{proof}
Let $\alpha$ be an arbitrary element in $\overline{S_{\widetilde{G}^{v_{i}}_{j}}}$. 
Due to the similar edge removal process, the proof is similar to that of Lemma \ref{lemma:2.1}. 
By similar arguments as in the proof of Lemma \ref{lemma:2.1}, we reduce to the
case $\alpha =\beta + \gamma$, where $\beta\in S_{\widetilde{G}^{v_{i}}_{j}}$, $\gamma = \mathbb{E}_{C} + \mathbb{E}_{C^{\prime}}$ for an exceptional pair $(C, C^{\prime})$ of $\widetilde{G}^{v_{i}}_{j}$.
Furthermore, we also get that $\alpha \in A$ if $\alpha_{w}=0$.
Assume that $\alpha_{w}>0$, then so is $\beta_{w}$. 
Hence among the edges defining the vector
$\beta$, there is at least one edge of the form $\{w,v\}$. 
Using Lemma \ref{lemma:1.2}, we get that the semigroup $S_{\widetilde{G}^{v_{i}}_{j}}$ contains $ \mathbb{E}_{C} + \mathbb{E}_{C^{\prime}} + \mathbf{e}_{w} +  \mathbf{e}_{v}$, hence it also contains $\alpha$.
\end{proof}

From the above observations, we have $S_{\widetilde{G}^{v_{i}}_{j}}\subset S_{\widetilde{G}^{v_{i}}_{j}}^{\prime}\subset \overline{S_{\widetilde{G}^{v_{i}}_{j}}}\subset  S_{\widetilde{G}^{v_{i}}_{j}}\cup A ,$ for all $ 1\leq i\leq b-2$, and $i\leq j\leq b-1$.

\begin{prop}\label{prop:3}
The edge ring $K[\widetilde{G}^{v_{i}}_{j}]$ of the graph $\widetilde{G}^{v_{i}}_{j}$ is non-normal and satisfies $(S_{2})$-condition, for all $1\leq i\leq b-2$ and $i\leq j\leq b-1$.
\end{prop}

\begin{proof}
As per our construction, the pair $(\{u_{a-2},u_{a-1},u_{a}\},\{v_{b-2},v_{b-1},v_{b}\})$ is contained in every $\widetilde{G}^{v_{i}}_{j}$ and is exceptional, for all $1\leq i\leq b-2$, and $i \leq j\leq b-1$. Hence, $K[\widetilde{G}^{v_{i}}_{j}]$ is always non-normal.

Let us consider an element $\alpha \in A$ such that $\alpha\in \overline{S_{\widetilde{G}_{j}^{v_{i}}}} \backslash S_{\widetilde{G}_{j}^{v_{i}}}$. 
The common vertex $w$ is regular in $\widetilde{G}^{v_{i}}_{j}$, and corresponding to this regular vertex, we have $S_{\widetilde{G}_{j}^{v_{i}}}\cap F_{w}:=\mathbb{Z}_{\geq 0 }\mathcal{A}_{\widetilde{G}_{j}^{v_{i}}\backslash w}$.
By a similar proof as that of Proposition~\ref{prop:2}, we can demonstrate that there exists no $\beta \in S_{\widetilde{G}_{j}^{v_{i}}}\cap F_{w}$, such that $\alpha +\beta \in S_{\widetilde{G}_{j}^{v_{i}}}$, and hence $S^{\prime}_{\widetilde{G}_{j}^{v_{i}}} = S_{\widetilde{G}_{j}^{v_{i}}}$. 
\end{proof}

Let the graph $\widetilde{G}^{v_{b-2}}_{b-1}:= \widetilde{G}$ (Figure~\ref{fig:rmvd}). We observe that, $\widetilde{G}$ is a subgraph of $G_{a,b}$ with $|V(\widetilde{G})|=|V(G_{a,b})|$ and $|E(\widetilde{G})|=a+b+2=d+1$. By Proposition~\ref{prop:3}, we know that the edge ring $K[\widetilde{G}]$ is non-normal and also satisfies $(S_{2})$-condition. Therefore, we observe that $\widetilde{G}$ is the graph on $d$ vertices with the least number of edges, $d+1$ edges, such that the edge ring is non-normal and meets $(S_{2})$-condition. This completes the proof of a part of the statement of Theorem \ref{thm:main}.

\section{Addition of edges to $G_{a,b}$ breaks non-normality or $(S_{2})$-condition}\label{sec:5}

In this section, we prove that any addition of (one or more) new edges to $G_{a,b}$ either breaks the non-normality of the edge ring or violates the $(S_{2})$-condition. 

Let us construct a new graph $G^{\prime}$ on the vertex set $V(G^{\prime})=V(G_{a,b})$, by introducing one or more edges to $G_{a,b}$. Since $K_{a+1}$ and $K_{b+1}$ are complete graphs, each of the new edges will be of the form $\{u_{i},v_{j}\}$, for some $1\leq i\leq a$ and $1\leq j \leq b$. For instance, addition of a single edge $\{u_{2},v_{3}\}$ to the graph $G_{a,b}$ is illustrated in Figure~\ref{fig:add}.

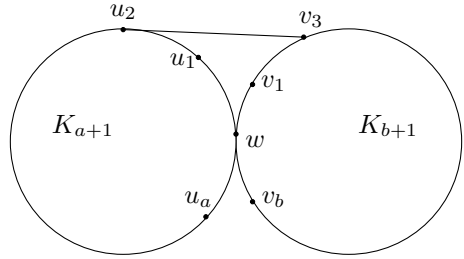
\begin{figure}[ht]
\centering
\begin{align*}
\begin{tikzpicture}
\filldraw[color=black!100, fill=white!5, thin](6,3) circle (1.5);
\filldraw[color=black!100, fill=white!5, thin](3,3) circle (1.5);
\filldraw [black] (4,4.12) circle (0.8pt);
\filldraw [black] (3,4.49) circle (0.8pt);
\filldraw [black] (5.4,4.39) circle (0.8pt);
\filldraw [black] (4.1,2) circle (0.8pt);
\filldraw [black] (4.72,2.2) circle (0.8pt);
\filldraw [black] (4.72,3.76) circle (0.8pt);
\filldraw [black] (4.5,3.1) circle (0.8pt);
\draw[black, thin] (3,4.49) -- (5.4,4.39);
\filldraw [black] (5,2) node[anchor=south] {$v_{b}$};
\filldraw [black] (5,4) node[anchor=north] {$v_{1}$};
\filldraw [black] (3.8,4.3) node[anchor=north] {$u_{1}$};
\filldraw [black] (4.5,3) node[anchor=west] {$w$};
\filldraw [black] (6,3.2) node[anchor=west] {$K_{b+1}$};
\filldraw [black] (3,4.5) node[anchor=south] {$u_{2}$};
\filldraw [black] (5.5,4.4) node[anchor=south] {$v_{3}$};
\filldraw [black] (3,3.2) node[anchor=east] {$K_{a+1}$};
\filldraw [black] (4,2) node[anchor=south] {$u_{a}$};
\end{tikzpicture}
\end{align*}
\caption{The graph $G^{\prime}$ obtained by adding edge $\{u_{2},v_{3}\}$ to $G_{a,b}$}
\label{fig:add}
\end{figure}

We can observe that for any $G^{\prime}$, all of its vertices are regular and the fundamental sets are: $\{i\},$ for some $  i\in V(G^{\prime})$  and $\{i,j\},\ \forall \ \{i,j\}\notin E(G^{\prime}).$

\begin{prop}\label{prop:add}
Let $G^{\prime}$ be a graph on the vertex set $V(G^{\prime})=V(G_{a,b})$, such that $G^{\prime}$ is constructed by adding one or more new edges to $G_{a,b}$. Then for any $G^{\prime}$, the edge ring $K[G^{\prime}]$ is either normal or it does not satisfy $(S_{2})$-condition.
\end{prop}

\begin{proof}

Suppose we construct a graph $G^{\prime}$ by adding at least minimal number of edges $\{u_{i},v_{j}\}$, where $1\leq i\leq a$ and $1\leq j \leq b$, such that we connect all the exceptional pairs of $G_{a,b}$. Thus, $G^{\prime}$ satisfies the odd-cycle condition and therefore the corresponding edge ring $K[G^{\prime}]$ is normal.

Now, let us consider the case where we construct a graph $G^{\prime}$ such that $K[G^{\prime}]$ is non-normal. Then, we prove that for any such $G^{\prime}$, the edge ring $K[G^{\prime}]$ will not satisfy $(S_{2})$-condition.

Suppose, we construct $G^{\prime}$ by adding new edges $\{u_{i},v_{j}\}$ to the graph $G_{a,b}$, for some $1\leq i\leq a$ and $1\leq j \leq b$, such that $G^{\prime}$ consists of at least one pair of $3$-cycles, $(\{u_{i_{1}},u_{i_{2}},u_{i_{3}}\},\{v_{j_{1}},v_{j_{2}},v_{j_{3}}\})$ with either $i_{k}\neq i$ or $j_{k}\neq j$ for any $1\leq k\leq 3$. This pair will be exceptional in $G^{\prime}$ and thus, the corresponding edge ring $K[G^{\prime}]$ is non-normal.  

Now, we consider any exceptional pair $(C,C^{\prime})$ of the graph $G^{\prime}$. For any regular vertex $v\in V(G^{\prime})\backslash [V(C)\cup V(C^{\prime})]$ such that $v\neq w$, we observe that $G^{\prime}\backslash v$ is a connected graph with the common vertex $w$. Let us consider the regular vertex $w\in V(G^{\prime})\backslash [V(C)\cup V(C^{\prime})]$. As per our construction, the graph $G^{\prime}$ contains edges of the type $\{u_{i},v_{j}\}$, for some $1\leq i\leq a$ and $1\leq j \leq b$. The existence of such edges in $G^{\prime}$ ensures the connectedness of the graph $G^{\prime}\backslash w$.

Thus for any  regular vertex $v\in V(G^{\prime})\backslash [V(C)\cup V(C^{\prime})]$, we observe that the graph $G^{\prime}\backslash v$ is always a connected graph. Hence both $C$ and $C^{\prime}$ belong to the same connected components of $G^{\prime}\backslash v$.

Since both $K_{a+1}$ and $K_{b+1}$ are complete graphs, any vertex in $V(K_{a+1})$ or $ V(K_{b+1})$ is adjacent to all the other vertices of $K_{a+1}$ and $K_{b+1}$ respectively.
Hence for all $v\in V(G^{\prime})$, we have $[V(C)\cup V(C^{\prime})] \cap [\{v\}\cup N(G^{\prime};\{v\}) ]\neq\emptyset.$ Let us consider the fundamental set of the form $\{u_{i},v_{j}\},$ such that $\{u_{i},v_{j}\}\notin E(G^{\prime})$.
By the completeness of $K_{a+1}$ and $K_{b+1}$, $\{u_{i},v_{j}\}\cup N (G^{\prime};\{u_{i},v_{j}\} )= V(G^{\prime}).$
Hence for any fundamental set $T$ of $G^{\prime}$, we have $$\Big[V(C)\cup V(C^{\prime})\Big]\cap \Big[ T \cup N\big(G^{\prime}; T \big)\Big]\neq\emptyset.$$

Therefore, by Theorem~\ref{thm:S2}, $\mathbb{E}_{C}+\mathbb{E}_{C^{\prime}}\in S_{G^{\prime}}^{\prime}$. In particular, $S_{G^{\prime}}\neq S_{G^{\prime}}^{\prime}$. Hence, the edge ring $K[G^{\prime}]$ does not satisfy $(S_{2})$-condition.
\end{proof}


\section{Conclusions}\label{sec:6}

A simple connected finite graph on $d$ vertices with a non-normal edge ring must contain at least one exceptional pair of odd cycles. Thus the minimal graph on $d$ vertices satisfying the above condition must be a graph consisting of two minimal odd cycles and a path (of at least length $2$) connecting the two cycles. This minimal graph will have exactly $d+1$ number of edges. In Section~\ref{sec:4}, we proved the existence of such a minimal graph $\widetilde{G}$, which satisfies the main theorem (Theorem~\ref{thm:main}).

We have examined the graph $G_{a,b}$ in detail. From Section~\ref{sec:5}, we can conclude that any addition of (one or more) new edges to $G_{a,b}$ either breaks the non-normality of the edge ring or violates $(S_{2})$-condition. Thus, we may conclude that $G_{a,b}$ is the graph on $d$ vertices with the maximum number of edges such that, the corresponding edge ring is non-normal and satisfies $(S_{2})$-condition. For the graph $G_{a,b}$, we have $|V(G_{a,b})|=d=a+b+1$ and $3\leq a\leq b$. Therefore, in order to maximize the number of edges in $G_{a,b}$, we have to consider $a=3$ and $b=d-4$. That is, $$\big\lvert E(G_{a,b})\big\rvert \leq \binom{4}{2}+\binom{d-3}{2}=\frac{d^{2}-7d+24}{2} .$$

This provides us very strong supporting evidence that $\frac{d^{2}-7d+24}{2}$ could be the maximal number of edges possible for a graph on $d$ vertices such that, its edge ring is non-normal and satisfies $(S_{2})$-condition.

\begin{proof}[\textbf{Proof of Theorem~\ref{thm:main}}]

Let us consider the graph $G_{3,b}$ on $d$ vertices such that $d\geq 7$. We have $|E(G_{3,b})|=\frac{d^{2}-7d+24}{2}$ and by Proposition~\ref{prop:1}, the edge ring $K[G_{3,b}]$ is non-normal and satisfies $(S_{2})$-condition. 

Through the edge removal processes discussed in Section~\ref{sec:4}, by eliminating one edge from the graph $G_{3,b}$ per step, we can gradually build up a graph on $d$ vertices with $d+1$ edges such that, its edge ring is non-normal and satisfies $(S_{2})$-condition. Proposition~\ref{prop:2} and Proposition~\ref{prop:3} guarantee that the edge ring of each of the graphs obtained after each removal step is always non-normal and will satisfy $(S_{2})$-condition.

 Therefore we prove that for any given integers $d$ and $n$ such that, $d\geq 7$ and $ d+1\leq n\leq\frac{d^{2}-7d+24}{2}$, we can always construct a finite simple connected graph on $d$ vertices and having $n$ edges such that, the edge ring of the graph is non-normal and satisfies $(S_{2})$-condition.
\end{proof}

\end{document}